\documentclass[11pt]{amsart}

\usepackage{amscd, amsfonts, mathrsfs, amsmath, amssymb, amsthm, mathtools}

\newtheorem{theorem}{Theorem}[section]
\newtheorem{corollary}[theorem]{Corollary}
\newtheorem{lemma}[theorem]{Lemma}

\theoremstyle{definition}

\def \Q {\mathbb{Q}}
\def \Qbar {\overline{\mathbb{Q}}}
\DeclareMathOperator{\rk}{rk}
\DeclareMathOperator{\Cl}{Cl}
\DeclareMathOperator{\Jac}{Jac}
\DeclareMathOperator{\ord}{ord}
\def \O {\mathcal{O}}

\numberwithin{equation}{section}

\title[Hilbert's Irreducibility Theorem and Ideal Class Groups]{Hilbert's Irreducibility Theorem and Ideal Class Groups of Quadratic Fields}
\author{Kaivalya Kulkarni}
\author{Aaron Levin}
\address{Okemos High School, 2800 Jolly Rd., Okemos, MI 48864, USA and \\
Department of Mathematics\\
Michigan State University\\
East Lansing, MI 48824\\
USA
}
\address{
Department of Mathematics\\
Michigan State University\\
East Lansing, MI 48824\\
USA
}
\email{kulkar80@msu.edu}
\email{adlevin@math.msu.edu}

\date{}

\begin{document}

\begin{abstract}
We prove a version of Hilbert's Irreducibility Theorem in the quadratic case, giving a quantitative improvement to a result of Bilu-Gillibert in this restricted setting. As an application, we give improvements to several quantitative results counting quadratic fields with certain types of ideal class groups. The proof of the main theorem is based on a result of Stewart and Top on values of binary forms modulo squares.
\end{abstract}

\maketitle 

\section{Introduction}

The Hilbert Irreducibility Theorem plays a key role in recent approaches to constructing and counting number fields with a large ideal class group (originating in work of the second author \cite{LevBordeaux} and joint work of the second author with Gillibert \cite{GL1}).  Recent applications of these techniques to study ideal class groups include work of Bilu-Gillibert \cite{BG}, A. Kulkarni \cite{Kul}, work of the second author with Gillibert \cite{GL2}, and work of the second author with Wiljanen and Yan \cite{LSW}.

Let $H(\alpha)$ denote the absolute multiplicative height of an algebraic number $\alpha$.  If $\alpha=p/q\in\mathbb{Q}$ is written in reduced form, then $H(p/q)=\max\{|p|,|q|\}$.
Partially in pursuit of applications to ideal class groups, Bilu and Gillibert \cite[Th.~3.1]{BG} proved the following version of Hilbert's Irreducibility Theorem (building on an enumerative result of Dvornicich and Zannier \cite{DZ}):

\begin{theorem}[Bilu-Gillibert {\cite[Th.~3.1]{BG}}]
\label{HIT2}
Let $k$ be a number field of degree $\ell$ over $\mathbb{Q}$.  Let $C$ be a curve over $k$ and $\phi:C\to\mathbb{P}^1$ a morphism (over $k$) of degree $d$.  Let~$S$ be a finite set of places of~$k$, $\epsilon>0$, and $\mho$ a thin subset of~$k$ \cite[\S 3.1]{BG}. Then there exist constants $B_0$ and $c$ such that for all $B \geq B_0$, among the number fields $k(P)$, where the point $P\in C(\bar k)$ satisfies 
\begin{align*}
\phi(P)&\in k\setminus \mho,\\
|\phi(P)|_v&<\epsilon, \qquad \forall v\in S,\\
H(\phi(P))&\le B,
\end{align*}
there exist at least $c B^\ell/\log B$ distinct fields of degree~$d$ over~$k$.
\end{theorem}

It was remarked by Bilu and Gillibert \cite[Rem.~3.2]{BG} that Theorem \ref{HIT2} likely holds with a lower bound of the form $cB^{2\ell}/(\log B)^A$ for some $A>0$. Our main result shows that this predicted lower bound holds (with $A=2$) when $k=\mathbb{Q}, \ell=1$, and $d=2$:

\begin{theorem}
\label{HIT}
Let $C$ be a curve over $\mathbb{Q}$ and $\phi:C\to\mathbb{P}^1$ a morphism of degree $2$.  Let~$S$ be a finite set of places of~$\mathbb{Q}$, $\epsilon>0$, and $\mho$ a thin subset of~$\mathbb{Q}$. Then there exist constants $B_0$ and $c$ such that for all $B \geq B_0$, among the number fields $\Q(P)$, where the point $P\in C(\Qbar)$ satisfies 
\begin{align*}
\phi(P)&\in \Q\setminus \mho,\\
|\phi(P)|_v&<\epsilon, \qquad \forall v\in S,\\
H(\phi(P))&\le B,
\end{align*}
there exist at least $cB^{2}/(\log B)^{2}$ distinct quadratic fields over~$\Q$.
\end{theorem}

Using Theorem \ref{HIT} in place of Theorem \ref{HIT2} allows us to recover and improve on several enumerative results involving ideal class groups of quadratic number fields. Given an integer $m>1$, it has been known since Nagell \cite{Nagell} that there are infinitely many imaginary quadratic number fields with class number divisible by $m$, and the analogous result for real quadratic fields was proved independently by Yamamoto \cite{Yam} and Weinberger \cite{Wein}. Quantitative results giving a lower bound for the number of such fields were given by Murty \cite{Murty}, Soundararajan \cite{Sound}, and Yu \cite{Yu}.

More generally, one can study the $m$-rank of the ideal class group. If $A$ is a finitely generated abelian group, we define the $m$-rank of $A$, $\rk_m A$, to be the largest integer $r$ such that $A$ has a subgroup isomorphic to $(\mathbb{Z}/m\mathbb{Z})^r$. For a number field $k$, we let $\Cl(k)$ denote its ideal class group and let $d_k$ denote its (absolute) discriminant.

As an application of Theorem \ref{HIT}, we first state a general result counting quadratic number fields with a large class group generated, via the technique of \cite{GL1}, from a hyperelliptic curve with a rational Weierstrass point and a large rational torsion subgroup in its Jacobian. The result is identical to \cite[Cor.~3.2]{GL1}, except that we improve the lower bound for an asymptotic count of such fields, by discriminant, from $X^{\frac{1}{2g+1}}/\log X$ to $X^{\frac{1}{g+1}}/(\log X)^2$ (up to a constant factor).

\begin{theorem}
\label{chyp}
Let $C$ be a smooth projective hyperelliptic curve over $\Q$ with a $\Q$-rational Weierstrass point. Let $g$ denote the genus of $C$ and let $\Jac(C)(\Q)_{\rm{tors}}$ denote the rational torsion subgroup of the Jacobian of $C$.  Let $m>1$ be an integer.  Then there exist $\gg \frac{X^{\frac{1}{g+1}}}{(\log X)^2}$ imaginary quadratic number fields $k$ with
\begin{equation*}
\rk_m \Cl(k)\geq \rk_m \Jac(C)(\Q)_{\rm{tors}},\quad |d_k|<X,
\end{equation*}
and $\gg \frac{X^{\frac{1}{g+1}}}{(\log X)^2}$ real quadratic number fields $k$ with
\begin{equation*}
\rk_m \Cl(k)\geq \rk_m \Jac(C)(\Q)_{\rm{tors}}-1,\quad d_k<X.
\end{equation*}
\end{theorem}

We let $\mathcal{N}^-(m^r;X)$ and $\mathcal{N}^+(m^r;X)$ denote the number of imaginary quadratic and real quadratic number fields $k$, respectively, with discriminant $d_k$ satisfying $|d_k|\leq X$ and such that $\rk_m\Cl(k)\geq r$.

By noting (see \cite[Lemma~3.3]{GL1}) that for $c\in\mathbb{Q}\setminus\{ 0,\pm 1\}$, the smooth projective hyperelliptic curve $C$ with affine equation 
\begin{align*}
y^2=x^{2m}-(1+c^2)x^m+c^2
\end{align*}
has genus $m-1$, a rational Weierstrass point, and $\rk_m \Jac(C)(\Q)_{\rm tors}\geq 2$, we find as a corollary:
\begin{corollary}
\label{chyp2}
Let $m>1$ be an integer.  Then
\begin{align*}
\mathcal{N}^-(m^2;X)&\gg X^{\frac{1}{m}}/(\log X)^2,\\
\mathcal{N}^+(m;X)&\gg X^{\frac{1}{m}}/(\log X)^2.
\end{align*}
\end{corollary}

When $m$ is odd, this yields a small improvement to results of Byeon \cite{Byeon} and Yu \cite{Yu} (following results of Murty \cite{Murty}), who proved $\mathcal{N}^-(m^2;X)\gg X^{\frac{1}{m}-\epsilon}$ and $\mathcal{N}^+(m;X)\gg X^{\frac{1}{m}-\epsilon}$, respectively. When $m$ is even, in the real quadratic case Chakraborty, Luca, and Mukhopadhyay \cite{CLM} (see also \cite{Luca}) proved the logarithmically better bound $\mathcal{N}^+(m;X)\gg X^{\frac{1}{m}}$.  When $m$ is even and $k$ is imaginary quadratic, Corollary \ref{chyp2} appears to be new and fills a gap in the literature, bringing this case in line with the other known results. 

For small values of $m$ better results are known (see results of Byeon \cite{Byeon2,Byeon3} for $m=5,7$). We discuss the case $m=3$, where Theorem \ref{HIT} again allows us to make quantitative improvements to some of the known results. In the case of class number divisibility by $3$, Heath-Brown \cite{HB} showed $\mathcal{N}^{\pm}(3;X)\gg X^{\frac{9}{10}-\epsilon}$, improving on \cite{BK, CM, Yu}. For $3$-rank $2$, Luca and Pacelli proved $\mathcal{N}^{\pm}(3^2;X)\gg X^{\frac{1}{3}}$, and recently Yu \cite{Yu2} improved this in the imaginary quadratic case, finding $\mathcal{N}^{-}(3^2;X)\gg X^{\frac{1}{2}-\epsilon}$.

For higher $3$-rank, it was shown by the second author and Wiljanen and Yan \cite{LSW} that $\mathcal{N}^{-}(3^3;X)\gg X^{\frac{1}{9}}/\log X$, $\mathcal{N}^{+}(3^4;X)\gg X^{\frac{1}{30}}/\log X$, and $\mathcal{N}^{-}(3^5;X)\gg X^{\frac{1}{30}}/\log X$. Using Theorem \ref{HIT}, we are able to improve these results.

\begin{theorem}
\label{3rankth}
We have
\begin{align*}
\mathcal{N}^{-}(3^3;X)&\gg \frac{X^{\frac{1}{5}}}{(\log X)^2},\\
\mathcal{N}^{+}(3^4;X)&\gg \frac{X^{\frac{1}{15}}}{(\log X)^2},\\
\mathcal{N}^{-}(3^5;X)&\gg \frac{X^{\frac{1}{15}}}{(\log X)^2}.
\end{align*}

\end{theorem}

The proof of our main theorem (Theorem \ref{HIT}), given in Section \ref{HITproof}, is based on a result of Stewart and Top \cite{ST} on the squarefree part of values of binary forms, which we describe in the next section. In the final section we briefly describe the proofs of the applications to ideal class groups (Theorem \ref{chyp} and Theorem \ref{3rankth}).

\section{Values of binary forms modulo $k$th powers}

The main tool in proving Theorem \ref{HIT} is a slight variation of a result of Stewart and Top \cite[Th.~2]{ST}.

\begin{theorem}[Stewart-Top]
\label{STth}
Let $A,B,M$, and $k$ be integers with $M\geq 1$ and $k\geq 2$. Let $F$ be a binary form with integer coefficients and degree $r$ which is not a constant multiple of a power of a linear form and which is not divisible over $\Q$ by the $k$th power of a non-constant binary form.  Let $S_{k}(x)$ denote the number of $k$-free integers $t$ with $|t|\leq x$ for which there exist positive integers $ a \leq x^{1/r}$ and $ b \leq x^{1/r}$ satisfying $a\equiv A \pmod{M}, b \equiv B \pmod{M}$, and $F(a,b) = tz^k$ for some nonzero integer $z$. Then 
\begin{align*}
S_{k}(x)\gg \frac{x^{\frac{2}{r}}}{(\log x)^2}.
\end{align*}
\end{theorem}

Stewart and Top's original result did not require the integers $a$ and $b$ to be positive, and did not place a bound on $a$ and $b$. Since the proof of Theorem~\ref{STth} only requires slight modifications to their original proof, we give a sketch of the proof highlighting the necessary changes.

\begin{proof}
Following Stewart and Top, we may write $F$ as a product $F_{1}F_{2}\cdots F_{l}$ of binary forms with integer coefficients, such that $F_{i+1}$ divides $F_{i}$ for all $1 \leq i \leq l - 1 $, and $F_i$ has nonzero discriminant for $i=1,\ldots, l$. Then $F_{1}$ may be written as a product of nonconstant forms $G_{1}\cdots G_m$ where $G_{i}$ is irreducible in $\Q[x,y]$ for all $i$.

The proof of Theorem 2 in \cite{ST} is divided into three cases: (1) some $G_{i}$ is nonlinear, (2) $G_1,\ldots, G_m$ are linear and $m\geq 3$, and (3)  $G_1,\ldots, G_m$ are linear and $m=2$.  In each case, we claim that the constructions used in \cite{ST} may be modified so that the integers $a$ and $b$ used are positive. 

In Case (1), $a$ and $b$ are positive integral linear combinations of integers $r_0,s_0,r_1,s_1$, where for a certain given lattice $\Lambda_p\subset\mathbb{Z}^2$, $(r_0,s_0)\in \Lambda_p$ is chosen such that $\max\{|r_0|,|s_0|\}$ is minimal, and $(r_1,s_1)$ is chosen such that $v_0=(r_0,s_0),v_1=(r_1,s_1)$ is a basis of $\Lambda_p$ and $\max\{|r_1|,|s_1|\}$ is minimal. Then we use the following lemma:
\begin{lemma}
Let $\Lambda\subset\mathbb{Z}^2$ be a lattice of rank $2$.  Let $(r_0,s_0)\in \Lambda$ be chosen such that $\max\{|r_0|,|s_0|\}$ is minimal, and $(r_1,s_1)$ chosen such that $v_0=(r_0,s_0),v_1=(r_1,s_1)$ is a basis of $\Lambda$ and $\max\{|r_1|,|s_1|\}$ is minimal.  Let $M=\max\{|r_0|,|s_0|,|r_1|,|s_1|\}$. Then there exists a basis $v_0'=(r_0',s_0'), v_1'=(r_1',s_1')\in \Lambda$ with $r_0',s_0',r_1',s_1
\geq 0$ such that
\begin{align}
\label{latticeineq}
\max\{|r_0'|,|s_0'|,|r_1'|,|s_1'|\}\leq 3\max\{|r_0|,|s_0|,|r_1|,|s_1|\}=3M.
\end{align}
\end{lemma}

\begin{proof}
Suppose first that $r_0$ and $s_0$ have the same sign or that $r_0s_0=0$. Then after possibly replacing $v_0$ by $-v_0$ and $v_1$ by $-v_1$, and after possibly interchanging the coordinates, we may assume that $r_0\geq s_0\geq 0$ and $r_1\geq 0$. If $s_1\geq 0$, then clearly we may take $v_0'=v_0$ and $v_1'=v_0$. So we may assume $s_1<0$. 

Note that $r_0\neq 0$ and let $n=\lceil r_1/r_0\rceil$. Let $v_0'=v_0$ and $v_1'=nv_0-v_1=(nr_0-r_1,ns_0-s_1)$. Then $v_0',v_1'$ are a basis of $\Lambda$. Since $r_1/r_0\leq n<r_1/r_0+1$ and $r_0\geq s_0$, we have
\begin{align*}
0\leq nr_0-r_1<r_0\leq M
\end{align*}
and
\begin{align*}
0\leq ns_0-s_1< r_1+s_0+|s_1|\leq 3M.
\end{align*}
Thus, we see that $v_0',v_1'$ satisfy the conclusions of the theorem.  The same proof, with the indices interchanged, works if $r_1$ and $s_1$ have the same sign or if $r_1s_1=0$.

Suppose now that $r_0r_1s_0s_1\neq 0$ and the coordinates of $v_i$ have opposite signs for $i=1,2$.  Then after possibly replacing $v_0$ by $-v_0$ and $v_1$ by $-v_1$, we may assume that $r_0,r_1>0$ and $s_0,s_1<0$. Then $v_1-v_0=(r_1-r_0,s_1-s_0)$ and $|r_1-r_0|< M$, $|s_1-s_0|<M$. Since $v_0, v_1-v_0$ is a basis of $\Lambda$, this contradicts the minimality of $v_0$ and $v_1$, and this case is impossible.
\end{proof}

Then in Case (1), modifying the construction to use the nonnegative integers $r_0',s_0',r_1',s_1'$ in place of $r_0,s_0,r_1,s_1$ preserves the conclusions in this case (with possibly slightly smaller constants).

In Case (2), the construction in \cite{ST} already uses positive integers $a$ and $b$. 

In Case (3), $a=A+kM, b=B+lcM$, and the parameters $A,B,c,l,M,t$ in the proof may all clearly be taken to be nonnegative. It only remains to show that the parameter $k$ in the proof may be constructed to be positive. If $d\leq 0$ then the existing proof already gives $k=t-dl>0$. If $d>0$ then we may replace $k=t-dl$ with $k=t+idl$ for some $i\in\{0,1\}$ satisfying $cf+ide\neq 0$, and the remainder of the proof remains substantially unchanged.

Finally, we note that all of the constructions produce integers $a$ and $b$ such that $\max\{|a|,|b|\}\ll x^{1/r}$. Then replacing $x$ by $cx$ for an appropriately small constant $c>0$, we see that we may choose $\max\{|a|,|b|\}\leq x^{1/r}$ in the constructions and, by the remainder of the proof in \cite{ST}, 
\begin{align*}
S_k(x)\gg \frac{(cx)^{\frac{2}{r}}}{(\log cx)^2}\gg \frac{x^{\frac{2}{r}}}{(\log x)^2}.
\end{align*}
\end{proof}

\section{Proof of Theorem \ref{HIT}}
\label{HITproof}

We need the following lemma for the proof of Theorem \ref{HIT}.

\begin{lemma}
\label{badprimes}
Let $S$ be a finite set of places of $\mathbb Q$, and let $\epsilon > 0$. Then there exists an invertible linear fractional transformation $\psi\in\Q(t)$ and integers $M$,  $A$, and $B$ such that whenever $\psi(t) = a/b$, $a,b>0$, $a\equiv A \pmod{M}, b \equiv B \pmod{M}$, we have $|t|_{v} < \epsilon$ for all $v\in S$.
\end{lemma}

\begin{proof}
If $S$ doesn't contain the (unique) archimedean place $\infty$ of $\mathbb{Q}$, then this is straightforward (with $\psi$ the identity).  If $S$ contains the archimedean place, let $N$ be an integer such that $N>1/\epsilon$ and let 
\begin{align*}
\psi(t)=\frac{1-Nt}{1+Nt}.
\end{align*} 
Then it is easily verified that $\psi^{-1}(0,\infty)=(-1/N,1/N)\subset (-\epsilon,\epsilon)$ and so if $\psi(t)=a/b>0$, then $|t|_\infty<\epsilon$. Since $\psi^{-1}(1)=0$, taking $A=B=1$ and $M$ divisible by sufficiently large powers of the (finite) primes in $S$, we see that if $a\equiv A \pmod{M}, b \equiv B \pmod{M}$ and $\psi(t)=a/b$, then $|t|_v<\epsilon$ for all finite places $v$ in $S$.
\end{proof}

We also need the following fact about thin sets \cite[p.~133]{Serre}:

\begin{lemma}
\label{thin}
Let $\mho\subset \Q$ be a thin set. Then for $x>0$, there exist at most $O(x)$ rational numbers $\alpha$ such that $\alpha\in\mho$ and $H(\alpha) < x$.
\end{lemma}

We now prove Theorem \ref{HIT}.

\begin{proof}[Proof of Theorem {\ref{HIT}}]
Let $C, \phi, S, \mho$, and $\epsilon$ be as in the statement of the theorem. Then $\phi$ induces a quadratic extension of function fields $\Q(C)/\Q(t)$ and we may write $\Q(C)=\Q(\sqrt{f(t)})$ for some nonconstant squarefree polynomial $f(t)\in \Q[t]$. Then $C$ may be taken to have affine equation $y^2=f(t)$, where $\phi$ is induced by the projection onto the $t$-coordinate.

Let $\psi, A, B$, and $M$ be as in Lemma \ref{badprimes} (with respect to $\epsilon$ and $S$). Let $\tau=\psi^{-1}$ and write $f(\tau(X/Y))=F(X,Y)R(X,Y)^2$ for some rational function $R(X,Y)\in \Q(X,Y)$ and squarefree homogeneous polynomial $F\in \mathbb{Z}[X,Y]$. Since $f$ is nonconstant squarefree, $F$ is nonconstant, and looking at degrees it follows easily that $\deg F$ is even. Thus, $\deg F\geq 2$. Then $F$ satisfies the hypotheses of Theorem \ref{STth} (with $k=2$).

Let $x$ be a positive real number. Let $T(x)$ be the set of squarefree integers $t$ for which there exist positive integers $ a \leq x$ and $ b \leq x$ satisfying $a\equiv A \pmod{M}, b \equiv B \pmod{M}$, and $F(a,b) = tz^2$ for some nonzero integer $z$. For each $t\in T(x)$, let $(a_t,b_t)$ be a pair of positive integers satisfying the conditions in the definition of $t\in T(x)$, and let $T'(x)=\{(a_t,b_t)\mid t\in T(x)\}$. Let 
\begin{align*}
R(x)=\{P\in C(\Qbar)\mid \psi(\phi(P))=a/b, (a,b)\in T'(x), \phi(P)\not\in \mho\}.
\end{align*} 

Let $P\in R(x)$. Then $\psi(\phi(P))=a_t/b_t=a/b$, for some $(a_t,b_t)=(a,b)\in T'(x)$ and $t\in T(x)$. From the definitions and Lemma \ref{badprimes}, we have 
\begin{align}
\label{maincond1}
\phi(P)=\tau(a/b)\in\Q\setminus \mho
\end{align}
and 
\begin{align}
\label{maincond2}
|\phi(P)|_v<\epsilon
\end{align}
for all $v\in S$. Note also that
\begin{align*}
\Q(P)&=\Q(\sqrt{f(\phi(P))})=\Q(\sqrt{f(\tau(a_t/b_t))})=\Q(\sqrt{F(a_t,b_t)R(a_t,b_t)^{2}})\\
&=\Q(\sqrt{F(a_t,b_t)})=\Q(\sqrt{t}).
\end{align*} 

It follows that the fields $\Q(P)$, $P\in R(x)$, are all distinct.  By elementary properties of heights, for some positive constant $c$ depending only on $\psi$,
\begin{align*}
H(\phi(P))\leq  cH(\psi(\phi(P)))
\end{align*} 
and $H(\psi(\phi(P)))=\max\{|a|,|b|\}\leq x$. Rescaling, if $P\in R(c^{-1}x)$ then $H(\phi(P))\leq x$ (and \eqref{maincond1} and \eqref{maincond2} hold). To finish the proof, it remains to count the elements in $R(c^{-1}x)$.

Let $\mho(x)=\{\alpha\in \mho\mid H(\alpha)<x\}$. By Lemma \ref{thin}, 
\begin{align*}
|\mho(x)|\ll x.
\end{align*}

Then by Theorem \ref{STth},
\begin{align*}
|R(c^{-1}x)|\geq |T'(c^{-1}x)|-2|\mho(x)|\geq S_2((c^{-1}x)^{\deg F})-2|\mho(x)|\gg \frac{x^2}{(\log x)^2},
\end{align*}
completing the proof.
 
\end{proof}

\section{Proof of Theorem \ref{chyp} and Theorem \ref{3rankth}}

We take the following result from the proof of \cite[Cor.~2.11]{GL1}:

\begin{theorem}
\label{GLth}
Let $C$ be a smooth projective curve over $\mathbb{Q}$, and let $m>1$. Let $S$ be the set of primes of bad reduction of $C$.  Let $\phi:C\to\mathbb{P}^1$ be a nonconstant morphism. Then there exists a thin set $\mho\subset \mathbb{Q}$ such that if $P\in C(\Qbar)$ and $\phi(P)\in \mathbb{Q}\setminus \mho$, then $[\Q(P):\Q]=\deg \phi$ and
\begin{align*}
\rk_m \Cl(\Q(P))\geq \rk_m  \Jac(C)(\Q)_{\rm{tors}}+\#S-\rk \O_{\Q(P),S}^*,
\end{align*}
\end{theorem}

Here, $\O_{\Q(P),S}^*$ denotes the group of $S'$-units of $\Q(P)$, where $S'$ consists of the set of places of $\Q(P)$ lying above $S$ along with the archimedean places. We now prove Theorem \ref{chyp}.

\begin{proof}[Proof of Theorem \ref{chyp}]
Since $C$ has a rational Weierstrass point, $C$ is birational to an affine curve given by an equation $y^2=f(x)$ with $f\in\mathbb{Z}[x]$ monic and $\deg f=d$ odd.  Let $S$ be the set of primes of bad reduction of $C$ and let $M=\prod_{p\in S}p$. Let $N$ be a large enough positive integer such that $f(x)>0$ if $|x-N/M|<1$ and such that $(M,N)=1$. Let $\phi:C\to\mathbb{P}^1$ be the morphism induced by $(x,y)\mapsto x-N/M$ and let $\mho$ be the thin set from Theorem \ref{GLth} (for $C$, $\phi$, and $m$). Let $P\in C(\Qbar)$ with $\phi(P)\in \Q\setminus \mho, |\phi(P)|_v<1$ for all $v\in S\cup \{\infty\}$, and $H(\phi(P))\leq B$.  Let $P=(x_0,y_0)$. Since $\phi(P)\in \Q$, we have $x_0\in \Q$.  For $p\in S$, by assumption $|x_0-N/M|_p<1$, and so $\ord_p x_0=\ord_p N/M=-1$. Since $f\in\mathbb{Z}[x]$ is monic, this implies that $\ord_p f(x_0)=-d$ is odd and $p$ ramifies in $\Q(P)=\Q(y_0)=\Q(\sqrt{f(x_0)})$.  By the choice of $N$, $|\phi(P)|_\infty<1$ implies that $f(x_0)>0$ and $\Q(P)$ is a real quadratic field. Therefore, by Dirichlet's unit theorem,
\begin{align*}
\#S-\rk \O_{\Q(P),S}^*=-\rk \O_{\Q(P)}^*=-1.
\end{align*}
Since $\phi(P)\in \Q\setminus \mho$, by Theorem \ref{GLth},
\begin{align*}
\rk_m \Cl(\Q(P))\geq \rk_m  \Jac(C)(\Q)_{\rm{tors}}-1.
\end{align*}

We now bound the discriminant of $\Q(P)$. First note that $H(\phi(P))\leq B$ implies $H(x_0)\leq cB$ for some constant $c$ (depending on $M$ and $N$). If $x_0=a/b$, then $\Q(\sqrt{f(x_0)})=\Q(\sqrt{b^{d+1}f(a/b)})$, where $b^{d+1}f(a/b)$ is a homogeneous polynomial of degree $d+1=2g+2$ in $a$ and $b$. It follows that $d_{\Q(P)}\leq c'B^{2g+2}$ for some constant $c'$ depending on $f$, $M$, and $N$. Setting $B=(X/c')^{\frac{1}{2g+2}}$ and using Theorem \ref{HIT}, we find  $\gg \frac{X^{\frac{1}{g+1}}}{(\log X)^2}$ distinct real quadratic fields $k=\Q(P)$ with $d_k<X$ and  $\rk_m \Cl(k)\geq \rk_m  \Jac(C)(\Q)_{\rm{tors}}-1$.  Finally, choosing $N$ to be a large enough positive integer such that $f(x)<0$ if $|x+N/M|<1$ (and $(M,N)=1$), and taking $\phi=x+N/M$, the same proof yields the result for imaginary quadratic fields (with the improvement over the real quadratic case coming from the difference in the ranks of the unit groups $\O_{\Q(P)}^*$).
\end{proof}

Finally, we give the proof of Theorem \ref{3rankth}.

\begin{proof}[Proof of Theorem \ref{3rankth}]
By \cite[Th.~6.1]{LSW}, the genus $4$ hyperelliptic curve 
\begin{align*}
C:y^2=t^{9} + 2973t^{6}-  369249t^3 + 11764900
\end{align*}
satisfies $\rk_3\Jac(C)(\Q)_{\rm{tors}}\geq 3$. Then the statement for $\mathcal{N}^{-}(3^3;X)$ follows immediately from the existence of this curve and Theorem \ref{chyp}. The proof of the statements involving $\mathcal{N}^{+}(3^4;X)$ and $\mathcal{N}^{-}(3^5;X)$ are identical to the proof of the main theorem in \cite{LSW}, except that one replaces the use of Theorem \ref{HIT2} in that proof with Theorem \ref{HIT}, providing the improvement in the bounds.

\end{proof}

\subsection*{Acknowledgments}
The second author was supported in part by NSF grant DMS-2001205.

\bibliographystyle{amsplain}
\bibliography{Hilbert}

\end{document}